\documentclass[a4paper]{amsart}
\usepackage[margin=3.5cm]{geometry}
\usepackage{amsmath}
\usepackage{amscd, amsfonts, amssymb, enumitem,mathrsfs, parskip}
\usepackage{tikz-cd} 

\usepackage{hyperref}
\usepackage[english]{babel}

\setlist[enumerate]{label=(\roman*)}

\newcommand\mm{\mathfrak m}

\newcommand{\R}{\mathrm{R}} 
\newcommand{\RR}{\mathscr{R}}
\renewcommand{\exp}{e}
\newcommand{\B}{\mathfrak B}
\newcommand{\D}{\mathscr D}

\newcommand\blfootnote[1]{%
  \begingroup
  \renewcommand\thefootnote{}\footnote{#1}%
  \addtocounter{footnote}{-1}%
  \endgroup
}

\DeclareMathOperator{\m}{m}
\DeclareMathOperator{\mr}{mr}
\DeclareMathOperator{\Em}{em}
\DeclareMathOperator{\emr}{emr}
\DeclareMathOperator{\E}{E}
\DeclareMathOperator{\Mat}{Mat}
\DeclareMathOperator{\diag}{diag}
\DeclareMathOperator{\GL}{GL}

\DeclareMathOperator{\SL}{SL}
\DeclareMathOperator{\SO}{SO}
\DeclareMathOperator{\Sp}{Sp}
\DeclareMathOperator{\PSp}{PSp}
\DeclareMathOperator{\PSO}{PSO}
\DeclareMathOperator{\Spin}{Spin}
\DeclareMathOperator{\HSpin}{HSpin}
\DeclareMathOperator{\Gm}{\mathbb{G}_m}
\DeclareMathOperator{\PGL}{PGL}

\numberwithin{equation}{section}

\newtheorem{thm}[equation]{Theorem}
\newtheorem{lem}[equation]{Lemma}
\newtheorem{cor}[equation]{Corollary}
\newtheorem{prop}[equation]{Proposition}

\theoremstyle{definition}
\newtheorem{defn}[equation]{Definition}

\theoremstyle{remark}
 
\newtheorem{rems}[equation]{Remarks}

\title[On the exponent of geometric unipotent radicals ]{On the exponent of geometric unipotent radicals of pseudo-reductive groups}
\author{Falk Bannuscher}
\address
{Fakult\"at f\"ur Mathematik, Ruhr-Universit\"at Bochum, D-44780 Bochum,Germany}
\email{falk.bannuscher@rub.de}
\author{Maike Gruchot}
\address{Lehrstuhl f\"ur Algebra und Zahlentheorie, RWTH Aachen University, D-52062 Aachen, Germany}
\email{maike.gruchot@rwth-aachen.de}
\author{David I. Stewart*}
\address{School of Mathematics, Statistics and Physics, Newcastle University, UK}
\email{david.stewart@newcastle.ac.uk}
\begin{document}
\begin{abstract}
Let $k'/k$ be a finite purely inseparable field extension and let $G'$ be a reductive $k'$-group.
We denote by $G=\R_{k'/k}(G')$ the Weil restriction of $G'$ across $k'/k$, a pseudo-reductive group. 
This article gives bounds for the exponent of the geometric unipotent radical $\RR_{u}(G_{\bar{k}})$ in terms of invariants of the extension $k'/k$, starting with the case $G'=\GL_n$ and applying these results to the case where $G'$ is a simple group.
\end{abstract}
\maketitle
\vspace{-50pt}\section{Introduction}
Let $k$ be a field, which we assume is imperfect of characteristic $p$. In this case there exist pseudo-reductive $k$-groups which are not reductive. Recall that a smooth, connected, affine algebraic $k$-group $G$ is \emph{pseudo-reductive} if the largest smooth connected unipotent normal $k$-subgroup $\RR_{u,k}(G)$ is trivial. \blfootnote{* Funded by Leverhulme Trust Research Project Grant number RPG-2021-080}

We are interested in the structure of the geometric unipotent radical $R:=\RR_u(G_{\bar k})=\RR_{u,\bar k}(G_{\bar k})$ of $G$. Since $R$ is a $p$-group, it makes sense to study its \textit{exponent}: the minimal integer $s$ such that the $p^s$-power map on the geometric unipotent radical factors through the trivial group. We denote the \emph{exponent} of $R$ by $e(R)$.

M.~Bate, B.~Martin, G.~R\"ohrle and the third author previously gave some bounds for $e(R)$ in \cite{BMRS19}. For example, \cite[Lem.~4.1]{BMRS19} implies that if $k'/k$ is a simple purely inseparable field extension with $(k')^{p^e}\subseteq k$, then $e(R)\leq e$.

The  monographs \cite{CGP15} and \cite{CPclassification16} contain a classification of pseudo-reductive groups. They are all related in some way or other to the Weil restrictions of reductive groups, which are themselves pseudo-reductive. We focus on groups which are Weil restrictions, since general pseudo-reductive groups contain central pseudo-reductive subgroups whose classification is thought to be intractable. Thus we assume unless stated otherwise that $G=\R_{k'/k}(G')$ for some reductive $k'$-group $G'$ where $k'$ is a finite non-zero reduced $k$-algebra. Since we are interested in $e(R)$, by the remarks in \cite[\S4]{BMRS19} we may as well consider the case that $k$ is separably closed and $k'$ is a purely inseparable field extension of $k$. In this case $G'$ is a split reductive group and as such is classified by its root datum. 

Recall that the exponent of the extension $k'/k$ is the smallest integer $e$ such that $(k')^{p^e}\subseteq k$. To describe our results, we need more sensitive data about the structure of $k'/k$. Since $k'/k$ is purely inseparable, we can appeal to the results in \cite{Pickert50} and \cite{Rasala71}, which guarantee a \emph{normal generating sequence} $\alpha_1,\ldots,\alpha_l$ of $k'$ over $k$ with certain properties---see Definition~\ref{dfn:ngs} below for more details. 
In particular, these elements come with certain exponents which are invariants of $k'/k$. These are the integers $e_1\geq \ldots\geq e_l$ such that $e_i$ is minimal subject to 
\[\alpha_i^{p^{e_i}}\in k(\alpha_1^{p^{e_i}},\ldots,\alpha_{i-1}^{p^{e_i}}).\]
It is immediate that $e_1=e$ and it follows from the Tower Law that $[k':k]=\prod p^{e_i}$. 
Let $r\in\mathbb{N}$. We define integers 
\begin{align*}\m(k'/k)&:=\sum_{i=1}^l(p^{e_i}-1)+1=\sum_{i=1}^{l}p^{e_i}-l+1,\\
\Em(k'/k)&:=\lceil \log_p(\m(k'/k))\rceil,\\
\mr(k'/k,r)&:=\sum_{i=1}^r p^{e_i}\text{,\quad where we take $e_i=0$ for $i>l$},\\
\emr(k'/k,r)&:=\lceil \log_p(\mr(k'/k,r))\rceil,\\
\text{and finally } \E(k'/k,r)&:=\min\{\Em(k'/k),\emr(k'/k,r)\}. \end{align*}
In other words $\Em(k'/k)$ is the minimal $s$ such that $p^s\geq \m(k'/k)$, and similarly for $\emr$.

In answering a question of \cite{BMRS19} on the exponent of the intersection of $R$ with a pseudo-Borel subgroup, we discovered that these numbers give lower bounds on $e(R)$ and for large enough rank determine the exponent of $R$ exactly. We first state our results in the case that $G'$ is a general linear group.
 
\begin{thm}
	\label{thm:1}
	Let $k'/k$ be a finite purely inseparable field extension of a separably closed field $k$ of characteristic $p$. Suppose $k'/k$ has exponent $e$ and a normal generating sequence $\alpha_1,\ldots,\alpha_l$ with sequence of exponents $e_1,\ldots,e_l$. For $r\in\mathbb{N}$, let $G':=\GL_r$, $G:=\R_{k'/k}(G')$ and $R:=\RR_u(G_{\bar k})$. Furthermore let $B'$ denote a Borel subgroup of $G'$ and $B:=\R_{k'/k}(B')\subseteq G$ the corresponding pseudo-Borel subgroup of $G$. The following hold:
	\begin{enumerate}
	\item if $s$ satisfies $p^s\geq p^e r$ then $e(R)$ is at most $s$; i.e.~$e(R)\leq \lceil e+\log_p(r)\rceil$;
	\item we have $e(R\cap B_{\bar k})=\E(k'/k,r)$;\label{thmpart:borel}
	\item in particular, $e(R)\geq \E(k'/k,r)$\label{thmpart:emr};
	\item if $\sum_{i=r+1}^l(p^{e_i}-1)<r-1$ then $e(R)=\E(k'/k,r)=\Em(k'/k)$\label{thmpart:exact}.
	\end{enumerate}
\end{thm}

\begin{rems}(i) The bounds in Theorem \ref{thm:1} are both arithmetic and Lie-theoretic in character: the integers $\m(k'/k)$ and $\Em(k'/k)$ depend only on the arithmetic structure of $k'/k$ but $\mr(k'/k,r)$, $\emr(k'/k,r)$ and $\E(k'/k,r)$ depend on both $k'/k$ and the rank $r$ of $G'$.

(ii). Part (i) of Theorem \ref{thm:1} improves \cite[Lem.~4.4]{BMRS19}, which required $p^s\geq r^2(p^e-1)$. As a corollary of (i) and (iv), if $G=\R_{k'/k}(\GL_r)$ for $p\geq r\geq 2$ then the exponent of $\RR_u(G_{\bar{k}})$ is $e$ if $k'/k$ is a simple extension and is $e+1$ otherwise. If $r=1$, then $e(R)=e$ as $\R_{k'/k}(\GL_1)$ is sent into the canonical $\GL_1$ $k$-subgroup by the $p^e$-power map, since it also sends $k'$ into $k$.

(iii). In \cite[Rem.~4.5(iv)]{BMRS19}, it was asked if the exponent $e(R)$ always coincides with the exponent $e(R\cap B_{\bar k})$. 
Clearly Theorem \ref{thm:1}(ii)-(iv) answers this question in the affirmative when $G'=\GL_r$ of large enough rank. 
We give an example to show that in general the answer is `no' (see Proposition \ref{prop:sl2}), coming from $\SL_2$ in characteristic $2$. In fact, in  Section~\ref{sec:exm} we give a complete description of $e(R)$ when $G'$ is $\SL_2$, $\PGL_2$ or $\GL_2$.
\end{rems}

Together with the rank $1$ results in Section~\ref{sec:exm}, the following (proved in Section~\ref{sec:rk2}) confirms the large rank behaviour in all simple groups, with mild conditions on the characteristic.

\begin{cor}\label{rankbig2} Let $k'/k$ be a finite, purely inseparable field extension of a separably closed field $k$ of characteristic $p$. Suppose $k'/k$ has exponent $e$ and a normal generating sequence $\alpha_1,\ldots,\alpha_l$ with sequence of exponents $e_1,\ldots,e_l$. Let $G'$ be a (split) simple algebraic $k$-group of rank $r\geq 2$, $G:=R_{k'/k}(G')$ and let $R:=\RR_{u}(G_{\bar k})$. If $G'$ is type $B,C,D,F$ then assume $p\neq 2$ or $G'$ is $\SO_{n}$ or $\Sp_{2n}$, and if $G$ is of type $E_6$ then assume $p\neq 3$. Then the exponent $e(R)\geq \E(k'/k,r)$; if $\sum_{i=r+1}^l(p^{e_i}-1)<r-1$ then in fact, $e(R)=\E(k'/k,r)=\Em(k'/k)$.\end{cor}

\section{Preliminaries and Notation}

Let $k$ be a field of characteristic $p$ and $G$ be a $k$-group, by which we mean an affine algebraic group scheme of finite type over $k$.
The \emph{$k$-unipotent radical} is the maximal smooth connected normal unipotent $k$-subgroup of $G$ and we denote this by $\RR_{u,k}(G)$.
The \emph{geometric unipotent radical} is the $\bar{k}$-unipotent radical of the base change $G_{\bar{k}}$ of $G$ to $\bar k$. This is denoted by $\RR_{u}(G_{\bar{k}})$.
We call a smooth connected $k$-group $G$ \emph{reductive} if $\RR_{u}(G_{\bar{k}})=1$ and \emph{pseudo-reductive} if $\RR_{u,k}(G)=1$. From now on $G$ will always be pseudo-reductive. We recall the definition and key properties of Weil restriction from \cite[A.5]{CGP15}.
\begin{defn}
Let $B\to B'$ be a finite flat map of noetherian rings, and $X'$ a quasi-projective $B'$-scheme. The \emph{Weil restriction} is a finite type $B$-scheme satisfying the universal property $$\R_{B'/B}(X')(A)=X'(B'\otimes_B A)$$ for all $B$-algebras $A$.
\end{defn}

The following proposition makes use of the natural map 
\begin{align}\label{eq:map-q}
q_{G'}:\R_{k'/k}(G')_{k'}\to G',
\end{align}
which is induced by $k'\otimes_k A \to A; a\otimes b \mapsto ab$ for any $k'$-algebra $A$. 

\begin{prop}[{\cite[Prop.~A.5.11, Thm.~1.6.2]{CGP15}, \cite[\S2]{BMRS19}, \cite[Lem.~3.4]{BMRS19}}]\label{prop:kernelqG} Let $k'/k$ be a  finite and purely inseparable field extension, $G'$ a non-trivial smooth connected $k'$-group and $G=\R_{k'/k}(G')$.
	\begin{enumerate}
		\item The kernel of $q_{G'}$ is a smooth connected unipotent $k'$-subgroup and thus it is contained in $\RR_{u,k'}(G_{k'})$.
		\item If $G'$ is reductive over $k'$, then the kernel of $q_{G'}$ has field of definition over k equal to $k'\subseteq \bar{k}$. Thus the kernel $\ker(q_{G'})$ coincides with $\RR_{u,k'}(G_{k'})$, which is a $k'$-descent of $\RR_u(G_{\bar{k}})$ (i.e.~$\RR_{u,k'}(G_{k'})_{\bar k}\cong \RR_u(G_{\bar{k}})$).
		\item If $H'$ is a reductive subgroup of $G'$ then the geometric unipotent radical of $\R_{k'/k}(H')$ is a subgroup of the geometric unipotent radical of $\R_{k'/k}(G')$.
		\item If $f:G_1\to G_2$ is an \'etale isogeny of reductive $k'$-groups then $\R_{k'/k}(f)_{\bar k}$ induces an isomorphism of geometric unipotent radicals of the respective Weil restrictions.
	\end{enumerate}	
\end{prop}
We are interested in the following invariant of a unipotent $k$-group:
\begin{defn}
	Let $k$ be a field of characteristic $p$ and $U$ a unipotent $k$-group.
	The exponent $\exp(U)$ of $U$ is the minimal $s$ such that the $p^s$-power map on $U$ factors through the trivial group. 
\end{defn}
Clearly, this definition is insensitive to base change. As in \cite[\S4]{BMRS19} we may therefore assume $k$ is separably closed for the remainder of the article.

When calculating with matrices, we make use of the fact that if $U(k)\subset U$ is dense (e.g.~if $k$ is perfect \cite[Thm.~34.4]{Humphreys75}), then a map factors through the trivial group if and only if it maps to the identity on $U(k)$. 
Since $R=\RR_u(G_{\bar k})$ descends to $R':=\RR_{u,k'}(G_{k'})$ whose $k'$-points are dense (by the fact that $k'$ is separably closed) we have $e(R)$ is the smallest $s$ such that $p^s$ kills $R'(k')$.

Following \cite{Oes84}, let $\B:=k'\otimes_k k'$. Then $\B$ is a local ring with maximal ideal  $$\mm=\langle 1\otimes x - x \otimes 1 \mid x\in k' \rangle=\ker(\phi),$$  where $\phi:k'\otimes_k k' \to k'; a\otimes b \mapsto ab$. Moreover, $\R_{k'/k}(G')_{k'}$ can be identified with $\R_{\B/k'}(G'_\B)$ and the $k'$-points of the $k'$-unipotent radical of the former is the kernel $G'(\mm)$ of the map $G'(\B)\to G'(k)$ induced by $\phi$. Suppose now $G'=\GL_r$ and $G=\R_{k'/k}(G')$. As the entries of any $M\in \Mat(r,\mm)$ are nilpotent elements of $\B$, it follows that $I_r+M$ has invertible determinant, hence is an element of $G'(\mm)$. Therefore \[\RR_{u,k}(\R_{k'/k}(\GL_r)_{k'})(k')=\{I_r+M\mid M\in \Mat(r,\mm)\}.\] 
As $I_r$ and $M$ commute, we have $A^{p^s}=I_r+M^{p^s}=I_r$ is equivalent to $M^{p^s}=0$. Hence:
\begin{lem}\label{lem:expglr} 
	The exponent of the geometric unipotent radical of $\R_{k'/k}(\GL_r)$ is the smallest $s$ such that $M^{p^s}=0$ for all $M\in \Mat(r,\mm)$.
\end{lem}

It was observed in \cite[Lem.~4.1]{BMRS19} that as $\mm^n=0$ for some $n\in \mathbb{N}$, one knows that $s$ in the lemma above is bounded above by the minimal $t$ such that $p^t\geq n$: the entries of $M^{p^t}$ are homogeneous polynomials of degree $p^t$ in elements of $\mm$ and so must vanish.

\begin{defn}\label{dfn:ngs}
	Let $k'/k''/k$ be a tower of finite purely inseparable field extensions.
	\begin{enumerate}
		\item If $x\in k'$ then the minimal $s\in \mathbb{N}$ such $x^{p^s}\in k''$ is called \emph{the exponent of $x$ with respect to $k''$} and denoted by $e_{k''}(x)$.
		\item We say that $x\in k'$ is \emph{normal in $k'/k''$} if $e_{k''}(x)\ge e_{k''}(y)$ for all $y\in k'$. 
		\item A sequence $\alpha_1,\ldots, \alpha_l$ in $k'$ is called a \emph{normal sequence} if for every $1\le i \le l$ and $k_i=k[\alpha_1,\ldots, \alpha_i]$ we have that $\alpha_i$ is normal in $k'/k_{i-1}$, and $\alpha_i \notin k_{i-1}$.
			\item A normal sequence $\alpha_1,\ldots, \alpha_l$ is called a \emph{normal generating sequence of $k'/k$} if $k[\alpha_1,\ldots, \alpha_l]=k'$.
		\item For a normal generating sequence $\alpha_1,\ldots, \alpha_l$ of $k'/k$ we set $e_i=e_{k_{i-1}}(\alpha_i)$. We call $e_1,\ldots,e_l$ a \emph{sequence of exponents for $k'/k$}
	\end{enumerate}
\end{defn}
Note that if $e_1,\ldots,e_l$ is a sequence of exponents for $k'/k$ then we have $e_1\ge e_2 \ge \ldots  \ge e_l$ due to \cite[Prop. 5]{Rasala71}. Crucially, by \textit{loc.~cit.} we also have \begin{equation}\label{eq:NGS} \alpha_i^{p^{e_i}}\in k(\alpha_{1}^{p^{e_i}},\ldots,\alpha_{i-1}^{p^{e_i}}) \text{ for all }1\le i\le l.\end{equation}
In particular, as the $e_i$ are decreasing,
\begin{equation}\label{eq:peipower} (k')^{p^{e_i}}\subseteq k(\alpha_1^{p^{e_i}},\ldots,\alpha_{i-1}^{p^{e_i}}).\end{equation}

Combining this with the main result of \cite{Pickert50}, we get that the sequence of exponents is an invariant of the field extension $k'/k$. More precisely:
\begin{lem}\label{lem:NGS}
	Let $k'/k$ be a finite purely inseparable field extension.
	Suppose that $\alpha_{1},\ldots,\alpha_l$ and $\beta_1,\ldots,\beta_{l'}$ are normal generating sequences of $k'/k$ with sequences of exponents $e_1,\ldots, e_l$ and $e'_1,\ldots, e'_{l'}$.
	Then $l=l'$ and $e_i=e'_i$ for  $1\leq i< l$.
\end{lem}

By the lemma, the numbers $\m(k'/k)$, $\mr(k'/k,r)$ defined in the introduction are both invariants of $k'/k$ and $(k'/k,r)$ respectively. To relate these with $\mm$, we first prove:

\begin{lem}\label{lem:mstuff}For a normal generating sequence $\alpha_1,\ldots, \alpha_l$ of $k'/k$, define $a_i:=1\otimes \alpha_i-\alpha_i\otimes 1\in\mm$. Then: 

\begin{enumerate}
	\item the ideal $\mm$ is generated by the $a_i$ over $k'$;
	\item the $p^{e_i}$-power map takes $\mm$ into the $k'$-subspace $k'[a_1^{p^{e_i}},\ldots,a_{i-1}^{p^{e_i}}]\cap \mm$.
\end{enumerate}
\end{lem}
\begin{proof} 
(i) Since $\mm$ is a codimension $1$ (left) $k'$-subspace of $\B$, it is a codimension $[k':k]$ $k$-subspace of $\B$. 	Define $J=\{ (b_1,b_2,\ldots,b_l)\in \mathbb{Z}^l  \mid  0\le b_i\le p^{e_i}-1 \}$.
For $b=(b_1,...,b_l)\in J$ we set $x^b=\prod_{i=1}^{l}\alpha_{i}^{b_i}$.
Then $\{x^b\}_{b\in J}$ is a $k$-basis of $k'$ and $\{1\otimes x^b\}$ is a $k'$-basis of $\B$. For $b\neq (0,\ldots, 0)$, the elements  $a_b:=1\otimes x^b-x^b\otimes 1=1\otimes x^b-x^b (1\otimes 1)$ are therefore linearly independent, generating a codimension $1$ subspace of $\B$, which is contained in the codimension $1$ kernel $\mm$ of the multiplication map $x\otimes y\mapsto xy$; hence they form a $k'$-basis of $\mm$. 
We wish to show $a_b\in \mm':=\mm\cap k'[a_1,\ldots,a_l]$ by induction on $\sum b_i$, the case $\sum b_i=1$ being trivial. 
So let $\sum b_i>1$. Without loss of generality, suppose $b_1>0$ and let $b'=b-(1,0,\ldots,0)$. Then inductively $a_{b'}=1\otimes x^{b'}-x^{b'}\otimes 1\in\mm'$. Now one checks directly that $a_b=x^{b'}a_1+\alpha_1x^{b'}-a_{b'}\cdot a_1\in\mm'$ as required. (And it follows that $\mm=\mm'$.)
		
(ii) Define $\mm':=k'[a_1^{p^{e_i}},\ldots,a_{i-1}^{p^{e_i}}]\cap \mm$. From (i), we can express any element of $\mm$ as a polynomial in $a_1,\ldots, a_l$ over $k'$. Since $k'$ has characteristic $p$, the $p^{e_i}$-power map distributes over addition and so it suffices to check that $(a_1^{r_1}\cdots a_l^{r_l})^{p^{e_i}}\in \mm'$ for $r_1,\ldots, r_l\in\mathbb{Z}_{\geq 0}$. For this it suffices to check that $a_j^{p^{e_i}}\in \mm'$ for $j\geq i$. Now $a_j^{p^{e_i}}=(1\otimes \alpha_j -\alpha_j\otimes 1)^{p^{e_i}}=1\otimes\alpha_j^{p^{e_i}}-\alpha_j^{p^{e_i}}\otimes 1$ and by (\ref{eq:peipower}) we can write $\alpha_j^{p^{e_i}}$ as a $k$-polynomial $P$ in $k[\alpha_1^{p^{e_i}},\ldots,\alpha_{i-1}^{p^{e_i}}]$ with no constant term. But applying part (i) with $l=i-1$, $a_j^{p^{e_i}}$ in place of $a_j$ for $1\leq j\leq i-1$ and $k[a_1^{p^{e_i}},\dots a_{i-1}^{p^{e_i}}]$ in place of $k'$, we deduce that $1\otimes P-P\otimes 1$ is indeed a polynomial in the $a_j^{p^{e_i}}$ with $1\leq j\leq i-1$ as required.\end{proof}

We may now identify the minimal integer $n$ such that $\mm^n=0$ with the integer $\m(k'/k)$. In fact we also prove something stronger:

\begin{lem}\label{lem:mbound}
With notation as above, let $d\leq l$ and suppose $m_1,\ldots,m_d\in\mm$ and $f_1,\ldots, f_d\in\mathbb{Z}_{\geq 0}$. 

If $\sum f_i\geq \mr(k'/k,d)-d+1=\sum p^{e^i}-d+1$ then $m_1^{f_1}\cdots m_d^{f_d}=0$. 

In particular, the minimum $n$ such that $\mm^{n}=0$ is $\m(k'/k)$.
\end{lem}
\begin{proof}
Assume for a contradiction that $m:=m_1^{f_1}\cdots m_d^{f_d}\neq 0$. We proceed by induction on $d$, observing that the product of any factors of $m$ is also non-zero. It does no harm to assume that $f_1\geq \ldots \geq f_d$. If $d=1$, then $\mr(k'/k,1)=p^{e_1}=p^e$ and indeed $m=0$, a contradiction. Otherwise, let $c$ be maximal such that $f_c\geq p^{e_c}$. This implies $\sum_{i=1}^c f_i\geq \mr(k'/k,c)-c+1$ so if $c<d$, then $m_1^{f_1}\cdots m_c^{f_c}$ is zero by induction, a contradiction. Therefore $c=d$. Let $(q_i,r_i)$ be the quotient and remainder when $p^{e_d}$ is divided into $f_i$. As $f_i\geq f_d\geq p^{e_d}$ we have each $q_i\geq 1$. Now \[f_1+\ldots+f_d=p^{e_d}(q_1+\ldots+q_d)+(r_1+\ldots+r_d)\geq p^{e_1}+\ldots+p^{e_d}-d+1\]
and so
\begin{align*}p^{e_d}(q_1+\ldots+q_d)&\geq p^{e_d}(q_1+\ldots+q_{d-1}+1)\\
&\geq p^{e_1}+\ldots+p^{e_d}-d+1-d(p^{e_d}-1)\\
&=p^{e_d}(p^{e_1-e_d}+\ldots+p^{e_{d-1}-e_d}-d+1)+1.\end{align*}
Since the second expression above is a multiple of $p^{e_d}$, we have \[q_1+\ldots+q_{d}\geq p^{e_1-e_d}+\ldots+p^{e_{d-1}-e_d}-(d-2).\]
Now as $q_ip^{e_d}\leq f_i$ for each $i$, we have $m'=m_1^{q_1p^{e_d}}\cdots m_d^{q_dp^{e_d}}\neq 0$. By Lemma \ref{lem:mstuff}(ii), we may write each $m_i^{p^{e_d}}$ as a polynomial in $k'[a_1^{p^{e_d}},\ldots,a_{d-1}^{p^{e_d}}]$. After expanding $m'$, any constituent monomial will be of the form $a_1^{\nu_1}\cdots a_{d-1}^{\nu_{d-1}}$ where $\sum_{i=1}^{d-1}\nu_i=p^{e_d}\sum_{i=1}^d q_i\geq p^{e_1}+\ldots+p^{e_{d-1}}-p^{e_d}(d-2)$ and each $\nu_i$ is divisible by $p^{e_d}$. Furthermore, as $a_{d-1}^{p^{e_{d-1}}}$ can be re-expressed in terms of the $a_i$ with $i<d-1$, we may assume $\nu_{d-1}\leq p^{e_{d-1}}-p^{e_d}$. It follows that 
\[\sum_{i=1}^{d-2}\nu_i\geq p^{e_1}+\ldots+p^{e_{d-2}}-p^{e_d}(d-3).\]
Continuing inductively, $\nu_1\geq p^{e_1}-p^{e_d}+1$. Thus $\nu_1\geq p^{e_1}$ and hence all monomials in $m'$ are zero, a contradiction.

We now tackle the second statement. To see that  $n\geq \m(k'/k)$, observe that $a:=\prod_{i=1}^{l}a_i^{p^{e_i}-1}$ is a product of $\m(k'/k)-1$ elements from $\{a_1,\ldots, a_l\}$. Expanding this product in terms of the $k$-basis $\{x^b\otimes x^c\mid b,c\in J\}$ of $\B$, the coefficient of $1\otimes x^b$ where $b=(p^{e_1}-1,\ldots,p^{e_l}-1)$ is easily seen to be $1$. It follows that $a\neq 0$.

For the upper bound, take $d=l$ in the first part of the lemma, noting that $\m(k'/k)=\mr(k'/k,l)-l+1$. 
\end{proof}

\section{Bounds in the case of the general linear group}

In this section, $G'=\GL_r$ and $G=\R_{k'/k}(G')$, where $k'/k$ is purely inseparable of exponent $e$. 
In \cite[Lem.~4.4]{BMRS19} it was shown that the exponent $e(R)$ of $R:=\RR_u(G_{\bar k})$ is at most $s$, where $s$ is chosen such that $p^s\geq r^2(p^e-1)$.
Using the Cayley--Hamilton theorem in its full generality (see \cite[III.8.11]{Bour98}) yields an improvement.
\begin{proof}[Proof of Theorem \ref{thm:1}(i)] 
	Let $M\in\Mat(r,\mm)$, as in Lemma \ref{lem:expglr}. 
	By the Cayley--Hamilton theorem, we have $\chi_M(M)=0$, where $\chi_M(\lambda)=\lambda^r+\sum_{t=0}^{r-1}f_t \lambda^t$ is the characteristic polynomial of $M$ with degree $r$. 
	Observe that the coefficient $f_t$ is either zero or a homogeneous polynomial in the entries of $\mm$ with degree $r-t$. 
	Since the matrices $M^{r-1},\ldots,M,I_r$ all commute and $k'\otimes_{k} k'$ has characteristic $p$, we get $$M^{rp^e}=-(f_{r-1}^{p^e}M^{(r-1)p^e}+\ldots+f_1^{p^e}M^{p^e}+f_0^{p^e}I_r)=0,$$
	since $f_t^{p^e}=0$ for all $f_t\in \mm$. 
	Choosing $s$ such that $p^s\ge p^er$ we get $M^{p^s}=M^{p^er}M^{p^s-p^er}=0$.
\end{proof}

Let $B'\subseteq G'$ be the upper Borel subgroup, whose points are upper triangular matrices, and $B=\R_{k'/k}(B')\subseteq G$ its Weil restriction---a pseudo-Borel subgroup in the terminology of \cite[C.2]{CPclassification16}. We now specify certain elements $M\in \Mat(r,\mm)$ with $M\in B(k')=B'(\B)$ to construct lower bounds on $e(R)$ which depend on $r$ and the exponents $e_i$ of a normal generating sequence.
Moreover, we calculate exactly the exponent of the intersection $R\cap B_{\bar k}$.

Before we give the proof, we make some elementary remarks about powers of matrices.
If $A$ is a commutative ring and $E_{i,j}\in\Mat(r,A)$ is the $r\times r$ elementry matrix with $e_{i,j}=1$ and zeros everywhere else, then we have $E_{i,j}E_{l,s}=E_{i,s}$ if and only if $j=l$ and $0$ else. Write $M=\sum a_{ij}E_{i,j}$ for a general element of $\Mat(r,A)$ . 
Then after expansion, any non-zero term in the power $M^n$ is of the form $a_w\cdot E_{i_1,i_{n+1}}$ where $a_w=a_{i_1,i_2}\cdot a_{i_2,i_3}\cdots a_{i_l,i_{n+1}}$, and $w\in (i_1,\ldots,i_{n+1})\in \{1,\ldots,r\}^{n+1}$.  
Furthermore, the $(i,j)$-th entry of $M^n$ is the sum over all $a_w$ such that $i_1=i$ and $i_{n+1}=j$. For $w\in \{1,\ldots,r\}^{n+1}$, say $n$ is the \emph{length} of $w$.
Now, if $M$ is an upper triangular matrix, then $a_w$ will be zero unless $i_j\leq i_{j+1}$ for each $1\leq j\leq n+1$. 
Hence the only $w$ yielding non-zero $a_w$ are non-decreasing sequences of integers from $\{1,\ldots,r\}$, whose length is bounded above by $\m(k'/k)$ (Lemma \ref{lem:mbound}).

\begin{proof}[Proof of Theorem \ref{thm:1}(ii)-(iv)]We work inside $B'(\mm)$. There are two cases. Suppose we have $\sum_{i=r+1}^l(p^{e_i}-1)<r-1$. Then there is an integer $q<r$ with $q$ maximal such that $(p^{e_{q+1}}-1)+\ldots+(p^{e_l}-1)\geq q-1$. Define the matrix $M\in B'(\mm)$ as follows. 

(i) Start with $M=0$;

(ii) From the top left corner, fill the first $q$ spaces on the diagonal with $a_1,\ldots, a_q$ respectively;

(iii) Fill up the first $q-1$ elements of the superdiagonal with $p^{e_l}-1$ entries of $a_l$, then $p^{e_{l-1}}-1$ entries of $a_{l-1}$ and so on;

(iv) If one has now written down $p^{e_{q+1}}-1$ entries $a_q$, then stop. Otherwise put $a_q$ in the $q$th position of the superdiagonal. If, after this, there are only $p^{e_{q+1}}-1-\tau$ entries $a_q$ with $\tau\geq 1$ on the superdiagonal, we set also the $(q+1)$st diagonal element equal to $a_q$.

We claim $M^{\m(k'/k)-1}\neq 0$. 
Take the sequence \[w=(\underbrace{1,\ldots,1}_{p^{e_1}\text{ times}},\underbrace{2,\ldots,2}_{p^{e_2}\text{ times}},\ldots, \underbrace{q,\ldots,q}_{p^{e_q}\text{ times}},\underbrace{q+1,\ldots,q+1}_{\tau+1\text{ times}})\] and observe that $a_w\neq 0$. 
A little thought shows that this is the unique sequence $w$ of this length with $a_w\neq 0$ starting at $1$ and ending at $q+1$.
This proves the claim. Hence we conclude $e(R\cap B_{\bar k})\geq \Em(k'/k)=\E(k'/k,r)$. Since $M^{\m(k'/k)}=0$ by Lemmas \ref{lem:mbound} and \ref{lem:expglr}, we have $e(R\cap B_{\bar k})\leq \Em(k'/k)$ and so this case is done.

Now suppose $\sum_{i=r+1}^l(p^{e_i}-1)\geq r-1$. One constructs a similar $M$. Start as before with $M=0$. One fills the diagonal with $a_1,\ldots, a_r$ and the superdiagonal with any $r-1$ elements from the list $(\underbrace{a_{r+1},\ldots,a_{r+1}}_{p^{e_{r+1}}-1\text{ times}},\ldots,\underbrace{a_l,\ldots,a_l}_{p^{e_l}-1 \text{ times}})$. 
This gives a matrix $M$ such that $M^m\neq 0$ where $m=\sum_{i=1}^r(p^{e_i}-1)+r-1=\mr(k'/k,r)-1$. Hence $e(R\cap B_{\bar k})\geq \emr(k'/k,r)$.

We need to demonstrate the upper bound for this case. Take a matrix $M\in B'(\mm)$ and any sequence $w$ of length $\mr(k'/k)$ such that $a_{w}\neq 0$ in a product of $\mr(k'/k,r)$ copies of $M$. Suppose the entries on the diagonal of $M$ are $m_1,\ldots,m_r\in\mm$. Then $a_w=m_1^{f_1}\cdots m_r^{f_r}\cdot\mu$ for non-negative integers $f_1,\ldots,f_r$ and $\mu$ a product of $s$ entries from above the diagonal; we have $s\leq r-1$. Since $\mr(k'/k)=\sum f_i+s$, we have $\sum f_i\geq \mr(k'/k,r)-r+1$. We are now done by Lemma \ref{lem:mbound}.

For part (iii) of the theorem, we note that any lower bound for $e(R\cap B_{\bar k})$ is automatically one for $e(R)$. For part (iv) we just observe that $\E(k'/k,r)=\Em(k'/k)$ is also an upper bound for $e(R)$ by Lemmas \ref{lem:mbound} and \ref{lem:expglr}.\end{proof}

\section{Extended example of \texorpdfstring{$\SL_2$}{SL2} and \texorpdfstring{$\PGL_2$}{PGL2}}
\label{sec:exm}
In this section we give a complete description of $e(R)$ in case $G'=\SL_2,\PGL_2$ or $\GL_2$.

\begin{prop} 
	Let $k'/k$ be purely inseparable of characteristic $p$ and exponent $e$. 
	Let $G$ be $\R_{k'/k}(\GL_2)$, $\R_{k'/k}(\PGL_2)$ or $\R_{k'/k}(\SL_2)$, the last of these only if $p\neq 2$. 
	Then $e(R)$ is $e$ if $k'/k$ is simple and $e+1$ otherwise.
\end{prop}
\begin{proof} 
	If $G'$ is $\GL_2$, then Theorem \ref{thm:1}(iv) gives the exponent $e(R)=e$ for $k'/k$ simple of exponent $e$, and at least $e+1$ when $k'/k$ is non-simple.	However, by Theorem \ref{thm:1}(i) we have $e(R)$ at most $e+1$. 
	As the map $\GL_2\to\PGL_2$ is smooth, and $\SL_2\subset\GL_2$, we immediately obtain $e$ or $e+1$ respectively as upper bounds for both the cases $G'=\SL_2$ and $G'=\PGL_2$. 
	In fact, the matrices constructed in the proof of Theorem \ref{thm:1}(ii)-(iv) and their relevant powers are clearly non-zero in their image in $\PGL_2(\B)$ and so the exponents of the radicals in case $G'=\GL_2$ and $G'=\PGL_2$ coincide. 
	Now by Proposition \ref{prop:kernelqG}(iv) these also coincide for $\SL_2$ outside of characteristic $2$, since the isogeny $f:\SL_2\to\PGL_2$ is smooth.
\end{proof}

We see some different behaviour for $\SL_2$ in characteristic $2$; in particular, the exponents of $R_u(G_{\bar{k}})\cap R_u(B_{\bar{k}})$ and $R_u(G_{\bar{k}})$ do not coincide.

Recall that a unital ring $R$ has characteristic $2$ if the equation $1+1=0$ holds in $R$.
We prove a general formula for $2^s$-th powers of matrices over rings of characteristic $2$. \begin{lem}
	\label{lem:formp=2}
	Let $R$ be a commutative ring of characteristic $2$ and $M=\begin{pmatrix}
	a & b \\
	c & d
	\end{pmatrix}$ with $a,b,c,d\in R$. Then for all $s\in \mathbb{N}_0$ we have
	\begin{align*}
	M^{2^s}=\begin{pmatrix}
	a^{2^s}+\sum_{i=0}^{s-1} b^{2^i}c^{2^i}t^{2^s-2^{i+1}} & bt^{2^s-1} \\
	ct^{2^s-1} & d^{2^s}+\sum_{i=0}^{s-1} b^{2^i}c^{2^i}t^{2^s-2^{i+1}}
	\end{pmatrix},
	\end{align*} 
	where $t=a+d$.
\end{lem}
\begin{proof}
	We argue by induction on $s$. For $s=0$, we just get $M$ on both sides.
	So assume the result for $s$ and set $l_s=\sum_{i=0}^{s-1} b^{2^i}c^{2^i}t^{2^s-2^{i+1}}$. 
	We note that 
	\begin{align*}
	l_s^2+bct^{2^{s+1}-2}=\sum_{i=1}^{s} b^{2^{i}}c^{2^{i}}t^{2^{s+1}-2^{i+1}}+bct^{2^{s+1}-2}=l_{s+1}
	\end{align*}
	holds. And so we get 
	\begin{align*}
	M^{2^{s+1}}
	=(M^{2^s})^2
	&=\begin{pmatrix}
	a^{2^{s+1}}+l_s^2+bct^{2^{s+1}-2} & bt^{2^s-1}(a^{2^s}+d^{2^s}) \\
	ct^{2^s-1}(a^{2^s}+d^{2^s}) & d^{2^{s+1}}+l_s^2+bct^{2^{s+1}-2}
	\end{pmatrix}\\
	&=\begin{pmatrix}
	a^{2^{s+1}}+l_{s+1} & bt^{2^{s+1}-1} \\
	ct^{2^{s+1}-1} & d^{2^{s+1}}+l_{s+1}
	\end{pmatrix}.\qedhere
	\end{align*}
\end{proof}

\begin{prop}\label{prop:sl2} 
Let $G'=\SL_2$, $p=2$ and suppose $k'/k$ has normal generating sequence $\alpha_1,\ldots,\alpha_l$ with exponents $e_1\geq e_2\geq \ldots \geq e_l$. Let $B'$ be an upper Borel subgroup of $G'$ and $B=\R_{k'/k}(B')$ a pseudo-Borel. 

(i) We have $e(R\cap B_{\bar k})=e$; 

(ii) We have $e(R)=\begin{cases}e+1\text{ if }e_1=e_2\\
e\text{ otherwise.}\end{cases}$\end{prop}

\begin{proof} We first calculate the exponent of $B'(\mm)=\ker B'(\B)\to B'(k')$. A general element is $I+M$ with $M=\begin{pmatrix}
a &b\\
0 & d\\
\end{pmatrix}$, where $a,b,d\in\mm=\langle a_1,\ldots,a_l\rangle$, satisfying $a+d+ad=0$ from the condition $\det(I+M)=1$. As $a\in\mm$, $1+a$ is invertible in $\B$, and so it follows that $d=a(1+a)^{-1}$ and $a+d=ad=a^2(1+a)^{-1}$. By Lemma \ref{lem:formp=2}, the $2^s$-th power of this matrix is \[\begin{pmatrix}
a^{2^s} &b(a^2(1+a)^{-1})^{2^s-1}\\
0 & d^{2^s}\\
\end{pmatrix}.\]
Choosing for example $a=a_1=1\otimes \alpha_1-\alpha_1\otimes1$, one sees the minimal $s$ for which $M^{2^s}=0$ is at least $e_1=e$ and also that if $s\geq e$ then the diagonal entries vanish. For $s=e$, the off-diagonal entry is a multiple of $(a^2)^{2^e-1}$, which has at least $2^e$ factors of $a$, hence is also. So $e(R\cap B_{\bar k})=e$, proving (i).

For (ii), a general element of $G'(\mm)$ is a matrix $I+M$ where $M=\begin{pmatrix}
a &b\\
c & d\\
\end{pmatrix}$, satisfying $a+d+ad+bc=0$, since $\det(I+M)=1$. 
If $l=1$, (i.e. $k'/k$ is simple of exponent $e=e_1$), then it is clear that $e(R)=e$. Suppose $e_2\leq e_1-1=e-1$. We may write $d=(1+a)^{-1}(a+bc)$, hence $t=a+d=(1+a)^{-1}(a^2+bc)$.  Therefore $t^{2^{e-1}}=(1+a)^{-2^{e-1}}(a^{2^e}+b^{2^{e-1}}c^{2^{e-1}})=(1+a)^{-2^{e-1}}(b^{2^{e-1}}c^{2^{e-1}})$. Now by Lemma \ref{lem:mstuff} we have $m^{2^{e-1}}\in \langle a_1^{2^{e-1}}\rangle\subseteq\mm$ for any $m\in\mm$. It follows that $b^{2^{e-1}}c^{2^{e-1}}$ has a factor of $a_1^{2^{e-1}}\cdot a_1^{2^{e-1}}=a_1^{2^e}=0$. By the same token, $t^{2^{e-1}}=0$ and hence, using the lemma, $M^{2^e}=0$. Therefore $e(R)=e$ as required.

Now suppose $e_1=e_2$. Since $G\subseteq \R_{k'/k}(\GL_2)$ we have $e(R)\leq e+1$. Set $b=a_1$, $c=a_2$, $d=0$ and $a=a_1a_2$. Then \[M^{2^e}=\begin{pmatrix}
(a_1a_2)^{2^e}+\sum_{i=1}^{e-1} (a_1a_2)^{2^i+2^e-2^{i+1}}  & a_1(a_1a_2)^{2^e-1}\\
a_2(a_1a_2)^{2^e-1} & \sum_{i=1}^{e-1} (a_1a_2)^{2^i+2^e-2^{i+1}}\\
\end{pmatrix}\]
The off-diagonal entries are zero, but $\sum_{i=1}^{e-1} (a_1a_2)^{2^i+2^e-2^{i+1}}=\sum_{i=1}^e (a_1a_2)^{2^e-2^{i}}$ and the $i=e-1$ entry is $(a_1a_2)^{2^{e-1}}$. Provided $e_2=e_1$, this element is non-zero by the same argument as in the proof of Lemma \ref{lem:mbound}---in a $k$-basis of $\B$, the coefficient of $1\otimes \alpha_1^{2^{e-1}}\alpha_2^{2^{e-1}}$ in $(a_1a_2)^{2^{e-1}}$ is evidently $1$. It follows that $M^{2^e}\neq 0$ and $e(R)=e+1$.
\end{proof}

\section{Simple groups of rank at least $2$}\label{sec:rk2}
For a positive integer $r$, recall that the algebraic $k$-group $\mu_r$ is the functor returning the $r$-th roots of unity in any $k$-algebra $A$; it is the kernel of the $r$-power map $\phi_r:\Gm\to\Gm;x\mapsto x^r$. (We have $\mu_r$ \'etale if and only if $p\nmid r$.) Say a $k'$-group $L$ is a \emph{skewed $\GL_n$} if $L\cong \GL_n/K$, where $K\cong\mu_r$, with $r\nmid n$; the quotient by $\mu_r$ induces the $r$-power map on the central torus of $\GL_n$. Since $K\cap\SL_n=1$, the derived subgroup $\D(L)$ of $L$ is still isomorphic to $\SL_n$ and notably there is still a central torus $C\subset L$ which contains $Z(\SL_n)$. In \cite[Ex.~5.17]{Tay19} it is explained that if $\mu_n\subset T$ is a finite subgroup of a split $k$-torus of rank at least $2$, then any automorphism of $\mu_n$ (which is some $r$-power map on $\mu_n$ with $(r,n)=1$) lifts to an automorphism of $T$. In our case, the quotient map by $K\times 1$ on $\GL_n\times \Gm$, lifts to an automorphism of $\GL_n\times \Gm$ defined on the central torus $\Gm\times\Gm$ by $(x,y)\mapsto (x^ry^{-n},x^by^a)$ where $a$ and $b$ are integers such that $ar+bn=1$. It follows that if $L$ is a skewed $\GL_n$-group, then $L\times \Gm\cong\GL_n\times\Gm$. 

The group $\SL_{n+1}$ contains a (Levi) subgroup $L:=\GL_n$ embedded block diagonally via $g\mapsto \diag(g,\det{g}^{-1})$. If $G'$ is a split simple $k$-group of type $A_n$ then $G'\cong \SL_{n+1}/\mu_r$ with $\mu_r\subseteq Z(\SL_{n+1})\cong\mu_{n+1}$, so $r\mid n+1$. But then $r\nmid n$, hence $L$ maps under the quotient by $\mu_r$ to a skewed $\GL_n$-subgroup of $G'$. This observation establishes the type $A$ case of the following lemma.

\begin{lem}\label{usefulsubgp}Let $G'$ be a split simple $k'$-group of rank $r$. Then one of the following holds:\begin{enumerate}\item $G'$ contains a skewed $\GL_r$ subgroup $L$;
\item There are \'etale isogenies $\pi:L\to L/Z$, and $\phi:\GL_r\to L/Z$ or $\phi:L\to\GL_r$;
\item $p=2$ and either $G'$ is of type $F_4$ or $G'$ is of type $B$, $C$ or $D$, but not isomorphic to  $\SO_{2r}$, $\SO_{2r+1}$ or $\Sp_{2r}$;
\item $p=3$ and $G$ is of type $E_6$. \end{enumerate}
In cases (i) and (ii) if we let $R=\RR_u(\R_{k'/k}(L)_{\bar k})$, then we have  $e(R)=e(\RR_u(\R_{k'/k}(\GL_r)_{\bar k}))$.
\end{lem}
\begin{proof}If $G'=\SO_{2r}$, $\SO_{2r+1}$ or $\Sp_{2r}$ then \cite[4.5]{Jan04} promises a Levi subgroup $L\cong\GL_r$ as required so (i) holds in all these cases. The central isogenies $\Sp_{2r}\to\PSp_{2r}$, $\Spin_n\to \SO_n$, $\SO_n\to \PSO_n$, $\Spin_{2r}\to \HSpin_{2r}$ or $\HSpin_{2r}\to\PSO_{2r}$ (where $\HSpin_{2r}$ occurs only if $n$ even) are all quotients by central subgroups isomorphic to $\mu_2$, hence are \'etale maps when $p>2$ and hence (ii) holds for all classical groups when $p>2$. Otherwise (iii) holds.

Let now $G'$ be exceptional. Note that groups of type $E_7$ and $E_8$ contain simple $A_7$ and $A_8$ maximal rank subgroups, evident from the Borel--de-Siebenthal algorithm. Hence each contains a skewed $\GL_7$ and $\GL_8$ subgroup $L$ respectively, satisfying (i) as required. By \cite[Thm.~A]{Kle88}, any Levi subgroup $L$ of $G_2$ is isomorphic to $\GL_2$ (rather than $\SL_2\times \Gm$, for example), hence satisfies (i). If $G'$ is of type $F_4$, then  we take $L$ to be an $A_3$ Levi subgroup of the $\Spin_9$ subgroup of $G'$. Now $L$ admits an \'etale isogeny with the $\GL_4$ Levi subgroup of $\SO_9\cong\Spin_9/Z(\Spin_9)$ if $p\neq 2$; so $L$ satisfies (ii) in that case, and if $p=2$ then $L$ is listed in (iii). It remains to deal with the case $G'=E_6$. Using Bourbaki notation \cite{Bourb02}, we take $L$ to be the $A_5$-Levi subgroup corresponding to the nodes $\{1,3,4,5,6\}$ of the Dynkin diagram explicitly. Assume first $G'$ is simply connected. Then $\D(L)$ is simply connected. Since $L$ centralises the $\SL_2$ subgroup generated by the $\pm\tilde\alpha$ root groups, where $\tilde\alpha$ is the longest root, the central torus $C$ of $L$ is the image of the cocharacter $h_{\tilde\alpha}:\Gm\to G'; t\mapsto h_{\tilde\alpha}(t)$. One calculates $h_{\tilde\alpha}(t)=h_{\alpha_1}(t)h_{\alpha_2}(t^2)h_{\alpha_3}(t^2)h_{\alpha_4}(t^3)h_{\alpha_5}(t^2)h_{\alpha_6}(t)$. This is inside $\D(L)\cong\SL_6$ iff $t^2=1$; therefore $\D(L)\cap C\cong \mu_2$---so $L$ is not a skewed $\GL_r$. However, if $p\neq 3$ then there is an \'etale isogeny from $L$ to $L/Z\cong \GL_6$ where $Z$ is a diagonal copy of $\mu_3$ in both factors $\D(L)$ and $C$ of $L$, so that (ii) holds. Besides which, since $Z(G')\cong\mu_3$, then if $p\neq 3$, the quotient by $Z(G')$ is \'etale and $L/Z(G')$ satisfies (ii) for the adjoint group $G/Z(G')$. Otherwise $p=3$ and (iv) holds.

We prove the last statement. Suppose (ii) holds. Then we note that if $\pi:A\to B$ is an \'etale central isogeny of reductive groups, then $\R_{k'/k}(A)\to\R_{k'/k}(B)$ is an isogeny whose kernel has trivial intersection with the unipotent radical. As in Proposition \ref{prop:kernelqG}(iv), it follows that the induced map $\RR_u(\R_{k'/k}(A)_{\bar k})\to\RR_u(\R_{k'/k}(B)_{\bar k})$ is an isomorphism as required. If (i) holds, then $L$ is a skewed $\GL_r$-subgroup and by the remarks preceding the lemma, $L\times \Gm\cong \GL_r\times \Gm$. Observe $\RR_u(\R_{k'/k}(L\times \Gm)_{\bar k})\cong\RR_u(\R_{k'/k}(L)_{\bar k})\times \RR_u(\R_{k'/k}(\Gm)_{\bar k})$; since $L$ already contains $C\cong\Gm$, the exponent of the product is equal to that of the first factor, $e(R)$. The same argument with $\GL_r\times \Gm$ in place of $L\times \Gm$ implies the claim.
\end{proof}

\begin{proof}[Proof of Corollary \ref{rankbig2}] With the hypotheses on $p$, by Lemma \ref{usefulsubgp} we get a reductive subgroup $L$ of $G'$, with \[e(\RR_u(\R_{k'/k}(L))_{\bar k})=e(\RR_u(\R_{k'/k}(\GL_r))_{\bar k})\geq \E(k'/k,r).\] By Proposition \ref{prop:kernelqG}(iii) it follows that $e(R)\geq \E(k'/k,r)$. If $\sum_{i=r+1}^l(p^{e_i}-1)<r-1$, then Weil restricting a faithful representation $G'\to \GL_n$ from $k'$ to $k$ gives an embedding of $G$ in $\GL_m$ with $m=[k':k]\cdot n\geq r$. Now by Proposition \ref{prop:kernelqG}(iii) again, $e(R)\leq e(\RR_u(\R_{k'/k}(\GL_m))_{\bar k})$ which is equal to $\E(k'/k,m)=\Em(k'/k)$ by Theorem \ref{thm:1}.\end{proof}

\textbf{Acknowledgments:} We thank Michael Bate and Gerhard R\"ohrle for helpful comments on early versions of the paper. Many thanks to Jay Taylor for discussions about root data. Finally, thanks to the referee for suggesting improvements to the exposition.
\bibliographystyle{amsplain}
\bibliography{exp}
\end{document}